\documentclass[12pt]{amsart}
\usepackage{ryanmacros}

\usepackage[colorlinks=true, pdfstartview=FitV, linkcolor=blue, citecolor=blue, urlcolor=blue]{hyperref}

\usepackage{fullpage}

\CompileMatrices

\DeclareMathOperator{\Pow}{Pow}

\newcommand{\oneto}[1]{\{1, \dotsc, #1 \}}
\newcommand{\genlines}[2]{L_{#1}^{#2}}

\newcommand{\db}[1]{e_{\{#1\}}^*}
\newcommand{\eb}[1]{e_{\{#1\}}}
\newcommand{\realcone}[1]{\mathcal{R}_{#1}}
\newcommand{\polymatcone}[1]{\mathcal{P}_{#1}}
\newcommand{\subv}{\mathcal{V}}
\newcommand{\ineq}[1]{I_{#1}}
\newcommand{\pb}{\#}
\newcommand{\pf}{\#}
\usepackage{stmaryrd}
\newcommand{\pair}[1]{\llbracket #1 \rrbracket}

\begin{document}

\title{New inequalities for subspace arrangements}
\author{Ryan Kinser}
\address{Department of Mathematics, University of Michigan, Ann Arbor, Michigan 48109}
\email{kinser@gmail.com}
\thanks{This material was based upon work supported under NSF Grant DMS 0349019}

\begin{abstract}
For each positive integer $n \geq 4$, we give an inequality satisfied by rank functions of arrangements of $n$ subspaces.  
When $n=4$ we recover Ingleton's inequality; for higher $n$ the inequalities are all new.
These inequalities can be thought of as a hierarchy of necessary conditions for a (poly)matroid to be realizable.
Some related open questions about the ``cone of realizable polymatroids'' are also presented.

\end{abstract}
%\comment{
\maketitle

%\tableofcontents

%\todo{prose, barvinok ref, check for dual pairing symbol again}

%%%%%%%%%%%%%%%%%%%%%%%%%%%%%%%%%%%%%%%%%%%%%%%%%%%
%								INTRODUCTION							%
%%%%%%%%%%%%%%%%%%%%%%%%%%%%%%%%%%%%%%%%%%%%%%%%%%%
\section{Introduction}

%%%%%%%%%%%%%%%%%%%%%%%%%%%%%%%%%%%%%%%%%%%%%%%%%%%
%					POLYMATROIDS AND ARRANGEMENTS							%
%%%%%%%%%%%%%%%%%%%%%%%%%%%%%%%%%%%%%%%%%%%%%%%%%%%
%\subsection{Polymatroids and subspace arrangements}
\subsection{Preliminaries}

For a set $X$, denote by $\Pow(X)$ the set of all subsets of $X$.  A \textbf{polymatroid} is a pair $(X, \rk)$, where $X$ is a finite set and
\[
\rk \colon {\Pow} (X) \to \N = \{0, 1, 2, \dotsc \}
\]
is a function satisfying:
\begin{enumerate}[({PM}1)]
\item $\rk(\emptyset) = 0$,
\item $\rk(A) \leq \rk(B)$  for $A \subseteq B$,
\item $\rk(A \cup B) + \rk(A \cap B) \leq \rk(A) + \rk(B)$ for all $A,B$.
\end{enumerate}
We call $X$ the \textbf{ground set} and $\rk$ the \textbf{rank function} of the polymatroid; sometimes we say that ``$\rk$ is a polymatroid on $X$.''
A \textbf{matroid} can be defined as a polymatroid for which the rank of each one element subset is  at most 1 \cite[\S~2.3]{MR849391}.
%\todo{more about matroids, generalities}
Two polymatroids $(X, \rk_X)$ and $(Y, \rk_Y)$ are said to be \textbf{isomorphic} if there exists a bijection $\varphi \colon X \to Y$ such that $\rk_Y \! \circ \varphi = \rk_X$.
In this paper we are only interested in $|X|$, the cardinality of $X$, so we can take $X=[n] := \oneto{n}$, and we write $\Pow(n):=\Pow([n])$.

A \textbf{subspace arrangement} is a collection of subspaces $\subv = \{V_1, \dotsc, V_n\}$ of some finite dimensional vector space.  Such a $\subv$ gives rise to a polymatroid $([n], \rk_\subv)$ by defining
\[
\rk_\subv(A) = \dim \left( \sum_{i \in A} V_i \right)
\]
(where the empty sum is 0).  
A polymatroid is said to be \textbf{realizable} (or representable) over a field $K$ if it is isomorphic to $([n], \rk_\subv)$ for some $K$-subspace arrangement $\subv$.  The unqualified statement ``$([n],\rk)$ is realizable'' is taken to mean that there exists some field over which $([n],\rk)$ is realizable, and this is the property of polymatroids that we will be interested in throughout this paper.   A general problem is to give combinatorial characterizations of realizability in various contexts (e.g., over a specific list of fields or over fields of given characteristics).
%Realizable matroids, also known as vector matroids, have been studied for decades.
%\todo{lame}
For example, there are explicit forbidden minor characterizations for realizability over $\mathbb{F}_2$, realizability over $\mathbb{F}_3$, and realizability over all fields \cite{MR849396} (here, $\mathbb{F}_q$ denotes the field with $q$ elements).
%\todo{more history, e.g. over F2 F3 etc ???}
%None of the results in this paper depend upon a choice of field, so we do not address this concern any further.

\subsection{The realizable cone}
Making the identification $\{F \colon \Pow(n) \to \R \} = \R^{2^{n}}$, we can consider the set of polymatroids on $[n]$ to be the integral points of a closed, convex cone $\polymatcone{n}$. 
%in the intersection of the half spaces $\{P(A) \geq 0\}_{A \in \Pow(X)}$.  
The notes \cite{Mustata:2005kx} and the book \cite{MR1940576} are good references for the elements of convex geometry.  
This cone is defined by the so-called \textbf{basic inequalities} (PM1), (PM2), and (PM3) above; 
%in the definition of a polymatroid.  
the set of realizable polymatroids on $[n]$
%\todo{too strong? just generates?}
%corresponds to the integer points on 
then generates a convex cone $\realcone{n} \subseteq \polymatcone{n}$.
%  We sometimes replace the subscript of these cones by $n:=|X|$ if we are not interested in the specific elements of $X$, and 
We refer to $\polymatcone{n}$ (resp. $\realcone{n}$) as the ``cone of polymatroids (resp. realizable polymatroids) on $n$ elements''; this could be somewhat misleading terminology, however (cf. \S\ref{sect:future}).
This viewpoint has been used by information theorists to study which polymatroids are obtained as the Shannon entropy of a discrete random vector \cite{Zhang:1998zr,Dougherty:2006yq,Matus:2007rt,Guille:2008sf}.
There doesn't seem to be much known about $\realcone{n}$ for $n > 4$;
%aside from the fact that $\realcone{n} \neq \polymatcone{n}$.  
see Section \ref{sect:future} for a discussion of open questions about $\realcone{n}$.
 
For $n \leq 3$, it is known that $\realcone{n} = \polymatcone{n}$, but this does not hold for $n\geq4$.  In the 1960s, A.W. Ingleton found that the following inequality is satisfied by any arrangement of four subspaces $\{V_1, V_2, V_3, V_4 \}$:
\begin{equation}\label{eq:ingleton}
\begin{split}
\dim(V_1 + V_2) + \dim(V_1 + V_3 + V_4) + \dim V_3 + \dim V_4 + \dim (V_2 + V_3 + V_4) \leq \\
\dim (V_1 + V_3) + \dim (V_1 + V_4) + \dim (V_2 + V_3) + \dim (V_2 + V_4) + \dim (V_3 + V_4) .
\end{split}
\end{equation}
This inequality does not follow from the defining inequalities for polymatroids (which can be seen by considering Vamos's matriod, cf. \cite[p.~159]{MR0278974}), so $\realcone{n} \neq \polymatcone{n}$ for $n \geq 4$.  A complete description of $\realcone{4}$ was given in \cite[Thm.~5]{MR1785025} by explicit computational methods. 
%any polymatroid satisfying
They found that the basic inequalities and all Ingleton inequalities, that is, those obtained by permutations of the indices in (\ref{eq:ingleton}), are enough to
% is realizable.
define $\realcone{4}$.
%\todo{is this actually true? or need to say something weaker?}
%Since this case is well-understood, we assume that $n>4$ for the remainder of the paper.
Ingleton asked (\emph{loc. cit.}) whether there might still be further independent inequalities satisfied by subspace arrangements; the following theorem, which is the main result of this paper, answers his question affirmatively.

\begin{theorem}\label{thm:main}
Let $V_1, \dotsc, V_n \subseteq V$ be a subspace arrangement with $n \geq 4$, and write
\[
\gen{i_1, \dotsc, i_r} := \dim \sum_{j=1}^r V_{i_j}.
\]
Then the inequality
\begin{equation}\label{eq:theorem}
\gen{1,2}  + \gen{1,3,n} +\gen{3} + \sum_{i=4}^n \bigl( \gen{i} + \gen{2,i-1,i} \bigr) \leq \gen{1,3} + \gen{1,n} + \gen{2,3} +\sum_{i=4}^{n} \bigl( \gen{2,i} + \gen{i-1, i} \bigr)
\end{equation}
holds, and is irreducible in the sense that it cannot be written as a sum of two nontrivial inequalities which hold for all subspace arrangements.  Furthermore, for each $n$, the inequality is independent of all inequalities which hold for fewer than $n$ subspaces.
\end{theorem}

The last statement is made more precise in Prop.~\ref{prop:nosub} using the constructions in the next section.
Note that by taking $n=4$ in (\ref{eq:theorem}), we recover Ingleton's inequality (\ref{eq:ingleton}), and for each $n>4$ we have a new necessary condition for a polymatroid to be realizable.
%\todo{disucss permutation sof the indices?}

\section{Operations on polymatroids and inequalities}\label{sect:operations}
Since every polymatroid lies in the hyperplane $H_n=\{F(\emptyset)=0\}$, we work in this subspace.  When thinking of the functions in $H_n$ as vectors in $\R^{2^n-1}$, we write $e_A$ for the standard basis vector which has a 1 in the coordinate indexed by $\emptyset \neq A \subseteq \oneto{n}$, and 0 elsewhere (i.e., the function which takes value 1 on $A$, and 0 on other subsets of $\oneto{n}$).  We write $\{e_A^*\}$ for the dual basis to $\{e_A \}$, and
\[
\pair{\ , \ } \colon H_n^* \times H_n \to \R
\]
for the standard pairing between $H_n$ and its dual vector space $H_n^*$.

The inequality (\ref{eq:theorem}) can be identified with the linear functional
\[
\ineq{n} = \db{1,3}+ \db{1,n} -\db{1,2} - \db{1,3,n} + \sum_{i=3}^n \left( \db{2,i} + \db{i-1,i} - \db{i} - \db{2, i-1, i} \right)
\]
on $H_n$, in that the inequality holds for a subspace arrangement $\subv$ if and only if $\pair{I_n, \rk_{\subv}} \geq 0$.  Recall that for any convex cone $\mathcal{C}$ in a vector space $V$, the dual cone $C^{\vee}$ is defined by
\[
\mathcal{C}^{\vee} := \setst{f \in V^*}{ \pair{f,c} \geq 0 \text{ for all }c \in \mathcal{C}} .
\]
Then the first two statements of Theorem \ref{thm:main} can be interpreted as saying that $I_n$ is an extremal ray of $\realcone{n}^{\vee}$ for any $n$.

For positive integers $k$ and $n$, a map%union preserving map
\[
\varphi \colon \Pow(k) \to \Pow(n)%	\qquad 	A \mapsto \bigcup_{a \in A} {\varphi}(a)
\]
such that $\varphi(\emptyset) = \emptyset$ induces linear maps
\begin{equation}\label{eq:sub}
\begin{split}
\varphi^{\pb} \colon &H_n \to H_k \qquad	\qquad \varphi_{\pf} \colon H_k^* \to H_n^* \\
%		&P \mapsto P \circ \phi	\qquad 	&	&f \mapsto \left(B \mapsto \sum_{A \in \phi^{-1}(B)} f(A) \right) \\
		&P \mapsto P \circ \varphi			\qquad \qquad \qquad	 e_A^* \mapsto e_{\varphi(A)}^* , 
\end{split}
\end{equation}
We define the first by thinking of elements of $H_n$ as functions, and the second using our standard dual basis.  
It is straightforward to check that these maps are dual to one another, so that
\begin{equation}\label{eq:adjoint}
\pair{f, \varphi^{\pb} P} = \pair{\varphi_{\pf} f, P}
\end{equation}
holds for any $f \in H_k^*$ and $P \in H_n$.  

We assume that all such maps between power sets appearing in this paper preserve unions (i.e., are morphisms of join semi-lattices), unless explicitly stated otherwise.  Such a $\varphi$ is order-preserving and completely determined by the images of one element sets; we write
\[
\varphi(i) := \varphi(\{i\})
\]
to simplify the notation.
The map $\varphi_{\pf}$ can be thought of as a ``substitution'' map for inequalities: for example, take $\varphi \colon \Pow(2) \to \Pow(3)$ determined by $\varphi(1) = \{1\},\, \varphi(2) = \{2,3\}$, and $f = \eb{1}^* + \eb{2}^* - \eb{1,2}^*$.  Then $\varphi_{\pf} f = \eb{1}^* + \eb{2,3}^* - \eb{1,2,3}^*$.

%\todo{duality in lemma 2 so simple?}
\begin{lemma}\label{lem:restrict}
With $\varphi$ as above, we have that 
\begin{enumerate}[(a)]
\item $\varphi^{\pb}$ restricts to maps $\polymatcone{n} \xto{\varphi^\pb} \polymatcone{k}$ and $\realcone{n} \xto{\varphi^{\pb}} \realcone{k}$, which are surjective when $\varphi$ is injective and injective when $\varphi$ is surjective;
\item $\varphi_{\pf}$ restricts to maps $\polymatcone{k}^{\vee} \xto{\varphi_{\pf}} \polymatcone{n}^{\vee}$ and $\realcone{k}^{\vee} \xto{\varphi_{\pf}} \realcone{n}^{\vee}$, which are surjective when $\varphi$ is surjective and injective when $\varphi$ is injective.
\end{enumerate}
%If $\varphi$ is injective, then the restriction of $\varphi^{\pb}$ to either $\polymatcone{n}$ $\realcone{}$above are surjective, and the restrictions of $\varphi_{\pf}$ are injective.  The dual statement is true if $\varphi$ is surjective.
\end{lemma}
\begin{proof}
Since $\varphi^{\pb}$ and $\varphi_{\pf}$ are dual to one another, we only need to prove the statements for $\varphi^{\pb}$.  The statements regarding the cones of polymatroids follow easily.  For example, that $\varphi^{\pb}X$ satisfies the submodularity condition (PM3) in the definition of a polymatroid is essentially equivalent to $\varphi$ preserving union.  For a realizable $\rk_{\subv} \in \realcone{n}$, where $\subv=\{V_1, \dotsc, V_n\}$, we have that $\varphi^{\pb} \rk_{\subv}$ is the rank function of the subspace arrangement
\[
\left\{ \sum_{i \in {\varphi}(1)} V_i , \dotsc , \sum_{i \in {\varphi}(k)} V_i \right\} .
\]

When $\varphi$ is injective, a realizable rank function $\rk_{\mathcal{W}} \in \realcone{k}$ (where $\mathcal{W} = \{W_1, \dotsc, W_k\}$) can always be lifted to some $\rk_{\subv} \in \realcone{n}$: for such a $\varphi$, each set $\varphi(i)$ contains at least one element $a_i$ which is not in any other $\varphi(j)$ (otherwise, we would have $\varphi(\oneto{k} \setminus \{i\})= \varphi(\oneto{k})$ and $\varphi$ would not be injective).  Fixing some choice of $\{a_i\}_{i=1}^k$, the subspace arrangement
\[
V_j =
\begin{cases}
W_i	& j=a_i	\\
0	& j \neq \text{any }a_i \\
\end{cases}
\]
satisfies $\varphi^{\pb} \rk_{\subv} = \rk_{\mathcal{W}}$, so $\varphi^{\pb}$ is surjective in this case.

The kernel of $\varphi^{\pb}$ is generated by $\setst{e_{A}}{A \notin \im \varphi}$, so $\varphi^{\pb}$ is injective when $\varphi$ is surjective, and so is its restriction to any subset of $H_n$.
\end{proof}

%We think of $\varphi^{\pb}$ as a ``forgetful'' map for polymatroids, and $\varphi_{\pf}$ is thought of as a ``substitution'' map for inequalities.

%%%%%%%%%%%%%%%%%%%%%%%%%%%%%%%%%%%%%%%%%%%%%%%%%%%
%					HIGHER INGLETON INEQUALITIES							%
%%%%%%%%%%%%%%%%%%%%%%%%%%%%%%%%%%%%%%%%%%%%%%%%%%%
\section{Proof of Theorem}\label{sect:higher}

\subsection{Validity of the inequalities}
First, we show that inequality (\ref{eq:theorem}) holds for an arrangement of $n$ subspaces.
%, and so by varying $n$ this gives a family of necessary conditions for a polymatroid to be realizable.  
%Then we demonstrate a hierarchical relation between the inequalities for successive values of $n$.

%V2 is the ``big'' space in this labeling.
\begin{proof}[Proof of the inequality (\ref{eq:theorem})]
%Note that the $i=1$ term in the sum simplifies to $[01]-[1]$.
Retain the notation of the statement of the theorem, and for any subspace $Z$ constructed from the $V_i$, we also denote by $\gen{Z }$ its dimension.
For a pair of subspaces $Y \subseteq Z$, write $[Z:Y] = \gen{Z}-\gen{Y}$.  We let the operation $+$ have precedence over $\cap$ in order to minimize the number of parentheses necessary.  So we have, for example, $A + B \cap C = (A + B) \cap C$.

Define $W=V_3 \cap \cdots \cap V_n$.  We have
\begin{equation}\label{eq:mainsubmod}
[W + V_1 + V_2 : W + V_1] \leq [W + V_2 : W]
\end{equation}
by submodularity. Starting with the left hand side, we have $\gen{W + V_1 + V_2} \geq \gen{1,2}$ by containment of subspaces.  Then using that $W + V_1 \subseteq V_3 + V_1 \cap V_n + V_1$, we find that
\[
\gen{W + V_1} \leq \gen{1,3}  + \gen{1, n} - \gen{1, 3, n}
\]
and so we get a lower bound for the left hand side of (\ref{eq:mainsubmod}):
\begin{equation*}
\gen{1,2}  -\gen{1,3} - \gen{1, n} +\gen{1, 3, n} \leq [W + V_1 + V_2 : W + V_1] .
\end{equation*}

On the right hand side, we have that $[W + V_2 : W] = [V_2 : V_2 \cap W]$.
Now we consider the descending chain of subspaces
\[
V_2 \supseteq V_2 \cap V_3 \supseteq \cdots \supseteq V_2 \cap \cdots \cap V_n = V_2 \cap W,
\]
which gives the formula
\begin{equation}\label{eq:tower}
 [V_2 : V_2 \cap W] = \sum_{i=3}^{n} \ [V_2 \cap \cdots \cap V_{i-1} : V_2 \cap \cdots \cap V_{i}].
\end{equation}
We give an upper bound on each summand of (\ref{eq:tower}):  for  $3 \leq i \leq n$, we have
\[
[V_2 \cap \cdots \cap V_{i-1} : V_2 \cap \cdots \cap V_{i}] = [V_{i} + ( V_2 \cap \cdots \cap  V_{i-1}) : V_{i}],
\]
and then using the containment $V_{i} + ( V_2 \cap \cdots \cap  V_{i-1}) \subseteq V_{i} + V_2 \cap V_{i} + V_{i-1}$ we find that
%\begin{equation}\label{eq:towerterm}
\[
 [V_{i} + ( V_2 \cap \cdots \cap  V_{i-1}) : V_{i}] \leq [V_{i} + V_2 \cap V_{i} + V_{i-1}: V_{i}] = \gen{2,i} + \gen{i-1, i} - \gen{2, i-1, i} - \gen{i}.
\]
%\end{equation}
Plugging this expression into (\ref{eq:tower}) and then into (\ref{eq:mainsubmod}) gives the main inequality (\ref{eq:theorem}) after rearranging.
\end{proof}

By varying $n$, the inequalities we obtain form a \textbf{hierarchy} in the following sense.  
%Write
%\[
%\ineq{n} = \db{1,3}+ \db{1,n} -\db{1,2} - \db{1,3,n} + \sum_{i=3}^n \left( \db{2,i} + \db{i-1,i} - \db{i} - \db{2, i-1, i} \right)
%\]
%for the linear functional on $H_n$ associated to the $n^{th}$ inequality.
If $\varphi \colon \Pow(n) \to \Pow(n-1)$ is given by
\[
{\varphi}(i) =
\begin{cases}
\{i\}	& i \neq n \\
\{1,n-1\}	& i = n ,\\
\end{cases}
\]
then it can be immediately verified that $\varphi_{\pf}$ takes the the inequality (\ref{eq:theorem}) for $n$ subspaces to the one for $n-1$ subspaces (i.e., $\varphi_{\pf}\ineq{n} = \ineq{n-1}$). 
% In this sense, the new inequalities form a hierarchy.  
\subsection{Independence of the inequalities}
Now we want to show that these inequalities are genuinely ``new'' in some appropriate sense.

\begin{prop}\label{prop:nosub}
The inequality (\ref{eq:theorem}) does not follow from a linear substitution into any inequality valid on a smaller number of subspaces.  More precisely, we have $\ineq{n} \notin \varphi_{\pf}(\realcone{k}^{\vee})$ for any $\varphi \colon \Pow(k) \to \Pow(n)$ with $k < n$.
\end{prop}
\begin{proof}
Suppose to the contrary that $\ineq{n} = \varphi_{\pf}f$ for some $f \in \realcone{k}^{\vee}$.  We will demonstrate a (non-realizable) polymatroid $T$ such that $\pair{\ineq{n}, T} = -1$ but $\varphi^{\pb}T$ is realizable (over any field).  Then using (\ref{eq:adjoint}) this would give $\pair{\ineq{n}, T} = \pair{\varphi_{\pf}f, T} = \pair{f, \varphi^{\pb}T} \geq 0$, a contradiction.

Consider the polymatroid $T \in \polymatcone{n}$ given by
\[
T(A) =
\begin{cases}
2	& A=\{2\};	\\
n-2	& A=\{i\}\ \text{with}\ i \neq 2;\\
n-1	& A=	\{2,i\}\ \text{or}\ \{i-1,i\}\ \text{with}\ i \geq 3,\ \text{or}\ \{1,3\},\ \text{or}\ \{1,n\};	\\
n	&\text{otherwise} .
\end{cases} 
\]
%\todo{(for $n=4$ this is dual to....}
Then $T$ is not realizable because $\pair{\ineq{n}, T} = -1$, but it is ``almost realizable'' in the sense that $\varphi^{\pb}T$ is realizable for any potential $\varphi$ that might give $\ineq{n} = \varphi_{\pf} f$.  To see this, we make some preliminary reductions.  Firstly, it is enough to consider the case $k=n-1$. This is because any $\varphi \colon \Pow(k) \to \Pow(n)$ factors (non-uniquely) as
\[
\varphi \colon \Pow(k) \xto{\varphi_1} \Pow(n-1) \xto{\varphi_2} \Pow(n) ,
\]
and so we have that
\[
\varphi^{\pb} T = T \circ \varphi = T \circ \varphi_2 \circ \varphi_1 = \varphi^{\pb}_1 \varphi^{\pb}_2 T .
\]
Thus, if $\varphi^{\pb}_2 T$ is realizable, then so is $\varphi^{\pb} T$ by Lemma \ref{lem:restrict}.

Now since each $\db{i}$ with $i \geq 3$ appears with nonzero coefficient in $\ineq{n}$, it must be that $\{i\} \in \im \varphi$ for $i \geq 3$.  So after possibly renumbering, we can assume that $\varphi (i) = \{i + 1\}$ for $i \geq 2$, and we need only consider the possible cases for $\varphi(1)$. 
Let $\mathcal{W} = \{W_i\}_{i=1}^{n-1}$ be a subspace arrangement in a vector space $W$, with basis $\{w_1, \dotsc, w_{n-1}, \tilde{w}\}$, defined as follows.  We take
\[
W_{i} = \gen{w_i, w_{i+1}, \dotsc, w_{i+n-4}, \tilde{w}} 
\]
for $i \geq 2$ (where the indices of the basis vectors are taken mod $n-1$), and $W_1$ will be chosen based on $\varphi(1)$ to make $\varphi^{\pb} T$ realizable.

In the case that $\varphi(1) = \emptyset$, it is straightforward to check that $\varphi^{\pb} T$ is realized by $\mathcal{W}$ when we take $W_1 = 0$.  If $\varphi^{\pb} T(1) = T \circ \varphi(1) = n$ (e.g., if $|\varphi(1)| \geq 3$), then  $\varphi^{\pb} T$ is realized by taking $W_1 = W$ instead.  Similarly, if $\varphi(1) \subseteq \{3, 4, \dotsc, n\}$, then we can take
\[
W_1 = \sum_{j \in \varphi(1)} W_j
\]
to realize $\varphi^{\pb}T$.

This leaves only the cases where $\varphi(1)$ is one of the sets $\{1\}, \{2\}, \{1,3\}, \{1, n\},$ or $\{2,i\}$ with $i \geq 3$, so we assume to be in this situation now. Consider the subspaces
\[
Z_1 = \gen{w_1, w_2, \dotsc, w_{n-3}, \tilde{w}}, \qquad Z_2 = \gen{w_1 + w_2 + \cdots + w_{n-1}, \tilde{w}}, 
\]
and the following table which associates a choice of $W_1$ to each remaining case for $\varphi(1)$.
\begin{center}
\begin{tabular}{|c|c|}
\hline
$\varphi(1)$ &	$W_1$	\\
\hline
$\{1\}	$&	$Z_1$\\
$\{2\}	$&	$Z_2$\\
$\{1,3\}$	& $Z_1 + W_2$	\\
$\{1,n\}$	& $Z_1 + W_{n-1}$	\\
$\{2,i\}$	& $Z_2 + W_{i-1}$	\\
\hline
\end{tabular}
\end{center}
Then using these choices of $W_1$, it can be checked that $\mathcal{W}$ realizes $\varphi^{\pb}T$ in each case, and this completes the proof.
\end{proof}

%%%%%%%%%%%%%%%%%%%%%%%%%%%%%%%%%%%%%%%%%%%%%%%%%%%
%							IRREDUCIBILITY								%
%%%%%%%%%%%%%%%%%%%%%%%%%%%%%%%%%%%%%%%%%%%%%%%%%%%
\subsection{Irreducibility of the inequalities}

If the term $\gen{1,2}$ is replaced by $\gen{1}$ in our new inequality (\ref{eq:theorem}), then the resulting inequality follows simply by adding together a collection of basic inequalities.  This might lead one to wonder how ``strong'' the new inequalities are.  We show in this subsection that (\ref{eq:theorem}) is \textbf{irreducible}, meaning that it cannot be written as a positive sum of any two nontrivial inequalities which hold for all subspace arrangements.

In the language of convex geometry, we show that $I_n$ defines a facet of $\realcone{n}$ (i.e., a face of codimension 1), which implies that $I_n$ is an extremal ray of $\realcone{n}^{\vee}$.
%or equivalently, that 
%and so $\ineq{n}$ is an extremal ray of the dual cone $\realcone{n}^{\vee}$.
%\todo{really? even if not polyhedral?}
%Fix $n \geq 4$ and let $e_S \in H_n$ denote the unit vector with 1 in the spot corresponding to $S$ and 0 elsewhere. 
%A restatement of Theorem \ref{thm:main} is that the linear functional $P_n$ is non-negative on $\realcone{n}$.
Fix $n>4$ and denote by $Z \subset H:=H_n$ the kernel of $I:=\ineq{n}$.  To show that $\realcone{} \cap Z$ is a facet of $\realcone{} :=\realcone{n}$, we need to show that $\dim (\realcone{} \cap Z) = \dim \realcone{} - 1 = 2^n - 2$.
(The fact that $\dim \realcone{} = 2^n-1$ is easy to see by considering arrangements of $n$ subspaces in a one dimensional ambient space; it also follows from the proposition below.)

For $S \subseteq [n]$ and $d\geq 1$, define $\genlines{S}{d} \in H$ as the polymatroid whose value on $A \subseteq [n]$ is
\[
\genlines{S}{d}(A) := \min \{d, |A \cap S| \} .
\]
This is the rank function of an arrangement of lines in general position in a $d$-dimensional vector space, where $\dim V_i = 1$ for $i \in S$ and $V_i =0$ otherwise, so $\genlines{S}{d} \in \realcone{}$.  It will be useful in the future to note that
\begin{equation}\label{eq:linessplit}
\genlines{S}{d} = \sum_{i \in S}\genlines{\{i\}}{1}
\end{equation}
for any $S \subseteq [n]$ such that $d \geq |S|$.  The following lemma gives some cases when $\genlines{S}{d} \in Z \cap \realcone{}$.

\begin{lemma}\label{lem:pnvanish}
The functional $I$ vanishes on $\genlines{S}{d}$ whenever $d \geq 3$.  If $d=2$, then it vanishes if
%\todo{check this}
\[
\begin{split}
\text{either }&2 \notin S \text{ and }\{1,3,n\} \nsubseteq S, \text{ or }\\
2 \in S \text{ and } S &\text{ contains no pair } \{i,i+1\} \text{ with $3 \leq i \leq n-1$} .
\end{split}
\]
If $d=1$, it vanishes when
\[
\begin{split}
\text{either }&2 \notin S \text{ and }S \cap \{1,3,n\} \neq \{3,n\}, \text{ or }\\
S=\{1,2,\dotsc,n \} &\text{ or } \{2,\dotsc,k \} \text{ or } \{2\} \cup \{k,\dotsc,n\} 
\end{split}
\]
for some $2 \leq k \leq n$.
\end{lemma}
\begin{proof}
The key is that $\genlines{S}{d}$ is the rank function for a set of lines $\{V_i\}$ in general position, and that dimensions add for direct sums of subspaces.  So when $d \geq |A|$, we get
\[
e_A^* (\genlines{S}{d}) = \dim \left(\sum_{a \in A} V_a \right) = \sum_{a \in A} \dim V_a 
%= \sum_{a \in A} \db{a}(\genlines{S}{d}) 
= \sum_{a \in A \cap S} 1 = |A \cap S|
\]
This gives the first statement by counting the appearances of each index on both sides of (\ref{eq:theorem}), since all terms $e_A^*$ appearing in $\ineq{}$ have $|A| \leq 3$.  The other two statements can be proven similarly, using some simple \emph{ad hoc} methods to account for the terms $e_A^*$ with $|A| > d$.
\end{proof}

\begin{lemma}\label{lem:idem}
The following hold in $H$:
\begin{align}\label{eq:genlineidem}
e_{[n]} 	&= \genlines{[n]}{n} - \genlines{[n]}{n-1} \\
e_S		&= \genlines{[n]}{n-1} - \genlines{S}{n-2} - \genlines{\{i\}}{1} \qquad \text{when }S= [n] \setminus \{i\} \\
e_S 		&= \sum_{A \supseteq S} (-1)^{|A \setminus S|+1} \genlines{A}{|A|-1} \qquad \text{for }|S| \leq n-2 .
\end{align}
%Consequently, $e_S \in F_n$ for $|S| \geq 4$.
\end{lemma}
\begin{proof}
For any $S \subseteq \oneto{n}$, we can directly compute from the definition that
\[
(\genlines{S}{|S|} - \genlines{S}{|S|-1})(A) = \min \{|S|, |A \cap S| \} - \min \{|S| -1, |A \cap S| \} =
\begin{cases}
1 & A \supseteq S \\
0 & \text{otherwise} ,
\end{cases}
\]
and so $\genlines{S}{|S|} - \genlines{S}{|S|-1} = \sum_{A \supseteq S} e_A$.  By applying M\"obius inversion \cite[\S~3.7]{stanleyenumcombin}, we can express each $e_S$ as
\[
e_S = \sum_{A \supseteq S} (-1)^{|A\setminus S|} (\genlines{A}{|A|} - \genlines{A}{|A|-1}) =  \sum_{A \supseteq S} (-1)^{|A\setminus S|} \genlines{A}{|A|} + \sum_{A \supseteq S} (-1)^{|A\setminus S|+1} \genlines{A}{|A|-1},
\]
since the M\"obius function of $\Pow(n)$ is $\mu(S,A) = (-1)^{|A \setminus S|}$ for $S\subseteq A$.
The first formula is immediately verified since the sums only consist of one term in this case.  When $S = [n]\setminus \{i\}$, we use that $\genlines{S}{|S|} - \genlines{[n]}{n} = -\genlines{\{i\}}{1}$ from (\ref{eq:linessplit}).
Finally, if $|S| \leq n-2$, then we can write the first sum as
\[
\sum_{A \supseteq S} \sum_{j \in A} (-1)^{|A \setminus S|} \genlines{\{j\}}{1} 
\]
using (\ref{eq:linessplit}) again. The coefficient of $\genlines{\{j\}}{1}$ in this sum is
\[
\sum_{A \supseteq S \cup \{j\}} (-1)^{|A \setminus S|} = %\sum_{A \supseteq S \cup \{j\}} (-1)^{|A \setminus S|} = 
(-1)^{n-|S|} \sum_{A \supseteq S \cup \{j\}} \mu(A,[n]), 
\]
which is 0 whenever $S \cup \{j\} \neq [n]$, by a basic property of the M\"obius function.
\end{proof}

\begin{lemma}\label{lem:repformulas}
The following hold in the vector space $H$:
\begin{align}
\genlines{\{i,j,k\}}{3} - \genlines{\{i,j,k\}}{2} &= \sum_{A \supseteq \{i,j,k\}} e_A \\
\genlines{T \cup \{a\}}{1} + \genlines{T \cup \{b\}}{1} - \genlines{T}{1} - \genlines{T \cup \{a,b\}}{1} &= \sum_{\substack{a,b \in A \\ A \cap T = \emptyset}} e_A .
\end{align}
\end{lemma}
\begin{proof}
The first equation is a special case of the first line of the proof of Lemma \ref{lem:idem}.
The second follows easily from the definitions.
\end{proof}

Let $F$ be the linear span of $\realcone{} \cap Z$.  To compute the dimension of $\realcone{} \cap Z$, we give an explicit basis of $F$.

\begin{prop}\label{prop:irred}
An explicit basis of $F$ is given by
\[
\setst{e_S + \alpha_S\, e_{\{1,3,n\}}}{S \neq \{1,3,n\}} ,
\]
where for $3\leq i \leq n$ and $3 \leq j \leq n-1$ we define
\[
\alpha_S =
\begin{cases}
-1	& S=\{i\}\text{ or }\{1,2\} \text{ or } \{2,j,j+1\}	\\
1	& S=\{1,3\}\text{ or }\{1,n\}\text{ or }S=\{2,i\}\text{ or } \{j, j+1\}	\\
0	& \text{otherwise} .
\end{cases} 
\]
Consequently, $\dim F=2^n-2$, and so $I$ defines a facet of $\realcone{}$ and the new inequalities (\ref{eq:theorem}) are irreducible.
\end{prop}
\begin{proof}
It is clear that the $2^n-2$ listed elements are linearly independent, hence we just need to show that each is in $F$.  We proceed by considering various cases for $S$.  Lemma \ref{lem:pnvanish} justifies the fact that all of the various $\genlines{S}{d}$ used in this proof are in $\realcone{} \cap Z$; we will not explicitly reference this fact at each occurence.  We also write $r:=\eb{1,3,n}$ to abbreviate the notation for the ``remainder'' term.
\begin{enumerate}[(1)]
\item If $|S| \geq 4$, then each term appearing in the expressions for $e_S$ in Lemma \ref{lem:idem} is in $\realcone{} \cap Z$, so $e_S \in F$.  In light of this, we ignore these ``higher'' terms in all sums below, without specific appeal to this item.
%  From Lemma \ref{lem:pnvanish}, we know that $\genlines{S}{d} \in W_n$ for $d \geq 3$, and so each term of the sums is in $F_n$ when $|S| \geq 4$.

\item  Suppose $|S|=3$.  Then from Lemma \ref{lem:repformulas} we have
\[
\genlines{S}{3} - \genlines{S}{2} = e_S +(\text{higher terms}) .
\]
The first summand on the left hand side is always in $F$, and the second is in $F$ unless $S = \{1,3,n\}$ or $S=\{2,l-1,l\}$ for some $4 \leq l \leq n$, so $e_S \in F$ except possibly in these two cases.

\item Now suppose $S=\{i,j\}$, so that taking $T=\emptyset$ in Lemma \ref{lem:repformulas} we get
\[
\genlines{\{i\}}{1} + \genlines{\{j\}}{1} - \genlines{\{i,j\}}{1} = \eb{i,j} + \sum_{k \neq i,j} \eb{i,j,k} + (\text{higher terms}) .
%\sum_{A \supseteq \{i,j\}} e_A .
\]
If $S \neq \{3,n\}$, then each term on the left hand side is in $F$, and so using (2) we get that $e_S \in F$ except possibly when $S \subseteq \{1,3,n\}$ or $S \subseteq \{2,l-1,l\}$ for some $4 \leq l \leq n$.

\item To deal with $S=\{3,n\}$, apply Lemma \ref{lem:repformulas} with $T=\{2,3,n\}^c,\, a=3,\, b=n$ to get
\[
\eb{3,n} + \eb{2,3,n} \in F .
\]
(here $A^c$ denotes the complement of $A$ in $\oneto{n}$).
Then (2), along with the assumption that $n > 4$, leaves that $\eb{3,n} \in F$.

\item A slight modification of the second part of Lemma \ref{lem:repformulas} (taking $T=\emptyset$ and replacing $\{a\}$ with $\emptyset$) shows that $\eb{1} = \genlines{\oneto{n}}{1} - \genlines{\{1\}^c}{1} \in F_n$, and similarly $\eb{2} \in F$.
This takes care of the cases that $\alpha_S =0$.

\item  The argument from (3), applied when $S=\{1,3\}$ and $S=\{1,n\}$, gives
\[
\eb{1,n} + r \in F \qquad \text{and} \qquad \eb{1,3} + r \in F ,
\]
as desired.

\item For $S=\{1,2\}$, we simply note that $\genlines{\{1\}}{1} = \sum_{A \supseteq \{1\}} e_A$, so using (2), (3), and (5) we get that $\eb{1,2} + \eb{1,3} + \eb{1,n} + r \in F$.  But then using (6) we can subtract
\[
\eb{1,n} + r + \eb{1,3} + r \in F
\]
to get that $\eb{1,2} - r \in F$.

\item Since $\genlines{\{1\}^c}{1} - \genlines{\{1,n\}^c}{1} = \eb{n} + \eb{1,n} \in F$, we get from (6) that $\eb{n} - r \in F$.

\item Similarly, we have $\genlines{\{2\}^c}{1} - \genlines{\{2,k\}^c}{1} = \eb{k} + \eb{2,k} \in F$ for $k \geq 3$.  In particular, the previous item gives that $\eb{2,n} + r \in F$.

\item Now take $S=\{2,k-1,k\}$ for some $4 \leq k \leq n$, and define $T:=\{1,\dotsc,k\}^{c}$.  Then we have
\[
\genlines{T \cup \{2\}}{1} + \genlines{T \cup \{k\}}{1} - \genlines{T}{1} - \genlines{T \cup \{2,k\}}{1} = 
%\sum_{\substack{2,l \in A \\ A \cap T = \emptyset}} e_A = 
\eb{2,k} + \sum_{l \notin \{2\} \cup \{k, \dotsc, n\}} \eb{2,k,l} + (\text{higher terms})
\]
from Lemma \ref{lem:repformulas}, and so by part (2) we get $\eb{2,k} + \eb{2,k-1,k} \in F$.  In particular, combining this with the previous item we find that $\eb{2,n-1,n} - r \in F$.

\item The reasoning from (3) shows that $\eb{k-1,k} + \eb{2,k-1,k}$ for any $4 \leq k \leq n$, then we use the previous item to get that $\eb{n-1,n}+ r \in F$.

\item We use (2) and (3) along with the expression
\[
\genlines{\{2,\dotsc,k\}}{1} -\genlines{\{2,\dotsc,k-1\}}{1} = \sum_{\{k\} \subseteq A \subseteq \{2, \dotsc,k-1\}^c} e_A
\]
to show that $\eb{k} + \eb{k,k+1} \in F$ for $3 < k < n$.  For $k=3$, we first get that $\eb{3} + \eb{1,3} + \eb{3,4} + \eb{3,n} + r \in F$ using (2) and (3), and then that $\eb{3} + \eb{3,4} \in F$ using (4) and (6).

\item Now using the terms from (12), (9), (10), and (11), in that order, we get the expression
\[
\begin{split}
\eb{k-1,k} + r = &(\eb{k,k+1} + r) - (\eb{k,k+1} + \eb{k}) + (\eb{k} + \eb{2,k})\\
 &- (\eb{2,k} + \eb{2,k-1,k}) + (\eb{2,k-1,k} + \eb{k-1,k})
\end{split}
\]
which shows that $\eb{k-1,k} +r \in F$ if and only if $\eb{k, k+1} + r \in F$, for any $4 \leq k < n$.  But in (11) we had that $\eb{n-1,n}+ r \in F$, so it must be that $\eb{k-1,k} + r \in F$ for all $4 \leq k \leq n$.
\end{enumerate}

This completes the proof of the proposition.
\end{proof}

All parts of Theorem \ref{thm:main} have now been proven.
%%%%%%%%%%%%%%%%%%%%%%%%%%%%%%%%%%%%%%%%%%%%%%%%%%%
%						THE FUTURE								%
%%%%%%%%%%%%%%%%%%%%%%%%%%%%%%%%%%%%%%%%%%%%%%%%%%%
\section{Future directions}\label{sect:future}

Of course, the most natural question to ask next is whether these new inequalities are enough to define $\realcone{n}$ for all $n$.  More precisely, we pose the following question.

\begin{question}
Are there any further independent inequalities on subspace arrangements besides those in this new hierarchy?
In other words, are there any more extremal rays of $\realcone{n}^{\vee}$ in addition to those obtained from:
\begin{enumerate}[a)]
\item substitutions into inequalities on fewer subspaces, that is, $\varphi_{\pf} f$ for some $\varphi$ and some $f \in \realcone{k}^{\vee}$ with $k < n$; and
\item all functionals obtained from $\ineq{n}$ by permuting the indices $\oneto{n}$?
\end{enumerate}
\end{question}

To try to answer this question, one could take the cone in $H_n^*$ defined by the functionals of a) and b), compute the extremal rays of the dual cone, and determine if these are realizable. This is what was done in \cite{MR1785025} for $\realcone{4}$.  However, this is beyond the author's computational power even for the smallest unknown case, $n=5$.  
An anonymous referee has suggested trying to reduce the problem by considering only connected polymatroids, which are those $(X, \rk)$ such that $\rk(X) - \rk(X \setminus S) = \rk (S)$ only when $S=X$ or $S = \emptyset$.  The cone generated by these should have fewer extremal rays, and inequalities on disconnected polymatroids should be understandable in terms of tensor products of inequalities on their connected components.

If the answer to the above question is ``no,'' then it would be nice to know if $\realcone{n}$ is even closed and/or polyhedral.
%Besides questions about the realizable cone itself, 
There are also questions related to the integral points of $\realcone{n}$, that is, those with integer coordinates.  Broadly, we can ask:
%``saturation'' problem to be addressed.
\begin{question}
Is every integral point of the cone $\realcone{n}$ (i.e., those with integer coordinates) a realizable polymatroid?
\end{question}
It is this question that makes the terminology ``realizable cone'' potentially misleading.
It can roughly be broken up into two questions.  First, we would like to know whether the sum of two polymatroids realizable over different fields is realizable (note here that this is the sum in the vector space $H_n$, not to be confused with the direct sum of matroids).  Second, there is a saturation problem:  if $P$ is a polymatroid such that there exists some $r \in \Q$ for which $rP$ is realizable, is $P$ itself realizable?
%\todo{finish this part}

%\todo{saturation for n=4? Compute small rational points with computer.}
 
\subsection*{Acknowledgments}
The author would like to thank Harm Derksen for suggesting the use of ``almost realizable'' polymatroids for finding new inequalities, specifically the one used here.  He would also like to thank Alan Stapledon for answering many questions about convex geometry, and Andreas Blass for directing him to the paper \cite{MR1785025} (which saved the author from including proof of an already known result).  Comments from two anonymous referees were also helpful in improving the quality of the manuscript.

\bibliographystyle{alpha}
\bibliography{ryanbiblio}

\end{document}